\documentclass[12pt]{amsart}
\usepackage{amsmath,amssymb,amsthm}

\usepackage[dvips]{graphicx}
\usepackage{color}
\usepackage{wasysym}
\newtheorem{lem}{Lemma}[section]

\newtheorem{thm}{Theorem}[section]

\theoremstyle{definition}

\theoremstyle{remark}

\theoremstyle{remark}
\newtheorem{remark}{Remark}[section]
\numberwithin{equation}{section}

\newcommand{\C}{{\mathbb C}}

\newcommand{\R}{{\mathbb R}}

\definecolor{blu}{rgb}{0,0,1}
\def\im{{\rm i}}

\newcommand{\vertiii}[1]{{\left\vert\kern-0.25ex\left\vert\kern-0.25ex\left\vert #1
    \right\vert\kern-0.25ex\right\vert\kern-0.25ex\right\vert}}

\begin{document}
\title[Traveling waves for the  quartic Half Wave]
{Traveling waves for the  quartic focusing Half Wave equation in one space dimension}
\author{Jacopo Bellazzini}
\address{J. Bellazzini,
\newline  Universit\`a di Sassari, Via Piandanna 4, 70100 Sassari, Italy}%
\email{jbellazzini@uniss.it}%
\author{Vladimir Georgiev}
\address{V. Georgiev
\newline Dipartimento di Matematica Universit\`a di Pisa
Largo B. Pontecorvo 5, 56100 Pisa, Italy\\
 and \\
 Faculty of Science and Engineering \\ Waseda University \\
 3-4-1, Okubo, Shinjuku-ku, Tokyo 169-8555 \\
Japan and IMI--BAS, Acad.
Georgi Bonchev Str., Block 8, 1113 Sofia, Bulgaria}%
\email{georgiev@dm.unipi.it}%
\author{Nicola Visciglia}
\address{N. Visciglia, \newline Dipartimento di Matematica Universit\`a di Pisa
Largo B. Pontecorvo 5, 56100 Pisa, Italy}%
\email{viscigli@dm.unipi.it}
\begin{abstract}
We consider the quartic focusing Half Wave equation (HW) in one space dimension. We show first that
that there exist  traveling wave solutions  with arbitrary small $H^{\frac 12}(\R)$ norm.
This fact shows that small data scattering is not possible for (HW) equation and that below the ground state energy there are solutions  whose energy travels as a localised packet and which preserve this localisation in time.  This behaviour   for (HW)  is in sharp contrast with classical
NLS in any dimension and with fractional NLS with radial data. The second result addressed is the non existence of traveling waves moving at the speed of light. The main ingredients of the proof are commutator estimates and a careful study of spatial decay of traveling waves profile using the harmonic extension to the upper half space.
\end{abstract}

\maketitle
Aim of this paper is to consider the Half Wave equation in one space dimension, (HW) since now on, with quartic nonlinearity
\begin{equation}\label{eq:main}
i \partial_t u =\sqrt{-\Delta }u - u|u|^{3},\quad (t,x)\in \R\times \R
\end{equation}
where $\sqrt{-\Delta }$ stands for the fractional Laplacian, namely  $\mathcal{F}(\sqrt{-\Delta }f)=|\xi|\mathcal{F}(f)$.
We recall that (HW) enjoy respectively
the conservation of the following energy:
\begin{equation}\label{eq:varprob2}
{\mathcal E}_{hw}(u)=\frac{1}{2}\|u\|_{\dot H^{\frac12}(\R)}^2-\frac{1}{5}\| u  \|^{5}_{L^{5}(\R)},
\end{equation}
as well as the conservation of the mass ${\mathcal M}(u)$, namely:
\begin{equation}\label{masscons} \frac d{dt} \|u(t, x)\|_{L^2(\R)}^2=0.
\end{equation}
Concerning the Cauchy problem associated with \eqref{eq:main} it has been proved in \cite{OV} the existence of global solutions for initial data  in $H^1(\R)$ and  $\sup_{(0,T)}||u||_{H^\frac 12(\R)}<\infty$,  for every $ T$. As a consequence of \eqref{eq:varprob2} and \eqref{masscons}
global existence is guaranteed for initial data in $H^1(\R)$ with small assumptions on $H^{\frac 12}(\R)$.\\
The aim of the paper is to prove  existence/non existence  results  for a class of solutions  whose energy travels as a localised packet and which preserve this localisation in time: the  \emph{traveling waves}. As a byproduct of our existence results and qualitative properties of traveling waves we show that small data scattering cannot occur for one dimensional quartic (HW).\\
To establish the existence of standing waves solutions $\psi_q(t,x)=e^{i t}q(x)$ for (HW) equation is not difficult to prove, see \cite{FrL}. The classical strategy introduced by Weinstein is to maximize a suitable functional whose critical points correspond to the standing waves. For (HW) equation one easily verifies that the map
$$u(t,x)\rightarrow \lambda_0^{\frac{1}{3}}u(\lambda_0 t+t_0, \lambda_0 x+x_0)e^{i \gamma_0}, \ \ (\lambda_0,t_0,x_0, \gamma_0)\in \R^{+}\times \R \times \R\times \R$$
yields a group of symmetries.
This implies trivially that if $\psi_q(t,x)=e^{i t}q(x)$ is the ground state solution to \eqref{eq:main} with mass $||q||_{L^2(\R)}^2=M_0$ and energy ${\mathcal E}_{hw}(\psi_q)=E_0$,  then $\psi_{q, \lambda_0}=\lambda_0^{\frac{1}{3}}q(\lambda_0x)e^{i \lambda_0 t}$ is another ground state solution with mass  $||\psi_{q, \lambda_0}||_{L^2(\R)}^2= \lambda_0^{-\frac{1}{3}}M_0$ and energy ${\mathcal E}_{hw}(\psi_{q, \lambda_0})= \lambda_0^{\frac{2}{3}}E_0$.  The aforementioned scaling  computation shows  that
\begin{equation}\label{eq:asymp}
{\mathcal E}_{hw}(\psi_{q, \lambda_0}){\mathcal M}(\psi_{q, \lambda_0})^{2}=const.
\end{equation}
In particular when the mass of the ground state goes to zero the energy of the ground state increses  as the inverse square of the mass. We recall that the term \emph{ground state} here indicates any positive standing wave solution to  \eqref{eq:main} even, nonnegative that vanish
at infinity. Fixing the value of $\lambda_0$ the ground state is unique, see \cite{FrL}. Notice that the ground state for quartic (HW) does not  minimize the energy with a mass constraint (indeed the energy for quartic (HW) is unbounded from below even with a mass constraint).\\
We shall underline that (HW) equation differs, for instance,  from NLS equation and nonlinear Klein-Gordon equation due to the lack of an explicit formula to construct traveling waves from standing waves. Indeed for NLS and NLKG equations  thanks to Galilean and  Lorentz invariance, the existence of traveling waves solutions can be straightforward obtained simply applying  a Galilean or Lorentz
boost to the standing waves. This is not the case for HW equation.\\ In \cite{FL2}, in the context of Boson star equation, i.e. considering a Hartree type nonlinearity, the existence
of traveling waves  for HW with the ansatz $u(x,t)=e^{i t}q_{v}(x-vt)$ has been investigated. More recently it has been proved in \cite{KLR} the existence of travelling waves  for the cubic one dimensional Half Wave equation with arbitrary small mass, i.e small $L^2$ norm.
This result  is important because it is a peculiarity of the (HW) equation. As an example of this peculiarity, for  classical $L^2$ critical NLS below the minimal mass ground state, all the initial data scatter to the free evolution equation \cite{D}.
Let us notice that traveling wave solutions $u(x,t)$ to \eqref{eq:main} with the ansatz $u(x,t)=e^{i t}Q_{v}(x-vt)$ fulfill the equation
\begin{equation}\label{eq:diffeq2}
\sqrt{-\Delta }Q_v +i v \partial_x Q_v+ Q_v- |Q_v|^{3}Q_v=0,
\end{equation}
and therefore the proof of their existence can be obtained maximizing again a suitable Weinstein-like functional, see e.g. \cite{KLR}.\\
For simplicity we will consider the case $v>0$, being the negative case identical.
The first result is   to prove the following theorem concerning the existence and the asymptotics when $v \rightarrow 1^{-}$
of traveling waves of the form $$\psi=e^{i  (1-v)t}Q_{v}(x-vt),$$ where $Q_v$ satisfies
\begin{equation}\label{eq:diffeq2a}
\sqrt{-\Delta }Q_v +i v \partial_x Q_v+  (1-v)Q_v- |Q_v|^{3}Q_v=0.
\end{equation}
 Notice that the dependence on $v$ is both on the phase and space shift.
\begin{thm}\label{maintheorem}
For any $0<v<1$ there exists $Q_{v}\in  H^{\frac{1}{2}}(\R)$ such that \newline $e^{i  (1-v)t}Q_{v}(x-vt)$ solves \eqref{eq:diffeq2a} and such that
\begin{equation}\label{eq:rel}
||Q_v||_{L^2(\R)} \sim (1-v)^{\frac 13},   \ \|Q_v\|_{H^{\frac 12}(\R)} \leq C (1-v)^{\frac 13}, \ \ \|Q_v\|_{H^{1}(\R)} \leq C (1-v)^{\frac 13}.
\end{equation}
In particular
 for $v$ sufficiently close to 1 the energy of the traveling wave is below the energy of the ground state with the same mass. Moreover, given $0 < v_1 < v_2 < 1$ with  $$v_1 = 1 - \varepsilon, \ \ v_2 = v_1 + \delta, \ \ 0 < \delta \ll \varepsilon \ll 1,$$ we have
 \begin{equation}\label{eq.rel1}
    ||Q_{v_1}-Q_{v_2}||_{L^2(\R)} \leq C \delta (1-v_1)^{\frac 13}.
 \end{equation}
In particular small data scattering does not occur.
\end{thm}
%

\begin{remark}
Notice that the lack of small data scattering should follow, as a matter of fact,  as a straightforward consequence of the existence of small traveling waves. Neverthless the rigorous proof of this implication requires some additional efforts since in our one dimensional context the (HW) equation does not enjoy nice decay estimates.
\end{remark}
\begin{remark}
The condition $0<v<1$ guarantees that the quadratic form associated with \eqref{eq:diffeq2a} fulfills for $u \neq0$
$$\|u\|_{\dot H^{\frac 12}(\R)}^2+ i  v \int_\R \bar u   \partial_x  u  dx> 0.$$
\end{remark}
We shall underline that our result confirms the peculiarity of (HW) with respect, for instance, to $L^2$ supercritical NLS.
For NLS, fixed the $L^2$ norm of the initial datum, and assuming that the energy of the wave is below the ground state energy, then the long time dynamics is characterized by only two possible alternatives:
either scattering to the free equation or blow up in finite time. The classical approach to show this alternative goes back to \cite{KM}. For (HW) we show that this situation cannot occur.  Moreover we underline that the dynamics of (HW) in one dimension is different from $L^2$ supercritical fractional NLS, i.e if one substitute the operator $\sqrt{-\Delta}$ with $(-\Delta)^{\frac{s}{2}}$ and $s>\frac 12$. Indeed for fractional NLS it is still true in dimension $n\geq 2$ and radial data the alternative  between scattering to the free equation or blow up in finite time for data with energy below the ground state energy, see \cite{BHL} for blow up and \cite{SZ} for scattering (only if $\frac{3}{4}<s<1$).
For a first attempt to describe the dynamics for (HW) equation in high dimension we quote \cite{BGV}. When $n\geq 2$ and radial data
the blow-up in finite time  is still an open question.\\
\\
The second contribution of this paper is to discuss the nonexistence of  traveling waves solutions  with arbitrary frequency moving  at the limit speed $v=1$. 

\begin{thm}\label{thm:nofast}
For any $\omega\in \R$ it does not exist a traveling wave solution  to \eqref{eq:main} given by
$u(x,t)=e^{i \omega t}Q_{1}(x-t)$ with $Q_1\in H^{\frac{1}{2}}(\R)$.
\end{thm}
\begin{remark}
We shall notice  that the non existence  of traveling wave solutions moving at  the speed of light is not elementary to prove. As a matter of fact  for $v=1$ the Fourier multiplier  $|\xi|-\xi$ in the kinetic term is no more positive  but it does not imply with elementary arguments that traveling waves cannot exist.\\
In fact the  crucial step in the proof is given by the spatial decay estimate of $Q_1$ (assuming the existence) together with commutator estimates. More precisely our approach is given by the following steps:

\begin{enumerate}
\item proving that any traveling wave at the speed of light decays $O(\frac{1}{|x|^2})$; 
\item thanks to (1) showing that at the speed of light supp $\hat Q_v\subset (-\infty, 0]$;
\item noticing that (2) implies that the Fourier multipliers are nothing but classical derivatives;
\item using the equation to conclude that $Q_v=0$.
\end{enumerate}
The proof of the decay of traveling waves at the speed of light is inspired by the celebrated work of Amick-Toland \cite{AT} in the context con Benjamin-Ono equation  (see also \cite{KMR}), while  the localization of frequences in the half space is a consequence of commutator estimates.\\
It is interesting to underline how   the \emph{non existence} of traveling waves at the speed of light for (HW) is  strongly correlated to the \emph{existence} of traveling waves with arbitrary speed for   Szeg\H{o} equation
\begin{equation}\label{eq:sz}
i \partial_t \psi =\Pi_{-}(|\psi|^3\psi),
\end{equation}
where $\Pi_{-}$ is the Szeg\H{o} projector onto negative frequencies. 
Indeed by  minimizing a suitable Weinstein functional a traveling wave solution $\psi(t,x)=e^{-i t}\tilde q_v(x-vt)$ for Szeg\H{o} equation  can be obtained for any $v< 0$, see \cite{PO} in the cubic case. These traveling waves can be rescaled in order to solve the equation
\begin{equation}\label{eq:sz1} 
 v D \tilde q_v + \tilde q_v- \Pi_{-}(|\tilde q_v|^{3}\tilde q_v)=0.
 \end{equation}
On the other hand the traveling waves at the speed of light  for (HW)  have frequency localization in the negative half space  and hence for those solutions 
the operator $\sqrt{-\Delta } +i  \partial_x $ coincides with $-2 D$.  We underline the interesting fact that for eq. \eqref{eq:sz1} and $v=-2$ solutions exist while solutions  at the speed of light  for (HW)  for $v=1$ does not. 
The key point in our non existing argument concerns the decay of traveling wave that implies, thanks to commutator estimates,  the frequency localization. The same argument cannot be applied for   Szeg\H{o} equation. In \cite{GLPR} it is noticed that 
while traveling waves moving at the speed of light for cubic Szeg\H{o} equation decay like $\frac{1}{<x>}$, the traveling waves with $0<v<1$ for cubic (HW) decay as $\frac{1}{x^2}$. Our aforementioned step (1) (Lemma \ref{lem: decay}) proves that at the speed of light  traveling waves for (HW) 
(independently on the value of exponent of the nonlinearity) shall decay as in case $0<v<1$. As a byproduct we proved that this decay implies, in fact, that such waves cannot exist. 
\end{remark}
\begin{remark}
It is immediate to notice that if  $u(x,t)=e^{i \omega t}Q_{1}(x-t)$ is a traveling wave solution at the speed of light with $Q_1\in H^{\frac{1}{2}}(\R)$, then $Q_1$ fulfills two additional properties: $\lim_{|x|\rightarrow \infty}|Q_1(x)|=0$ and $Q_1$ bounded. These two properties that follows from the fact that $Q_1\in H^1(\R)$ and are crucial to prove the non existence result.
\end{remark}
\section{A remark about travelling waves for NLS}
Consider the classical NLS
\begin{equation}\label{eq:mainNLS}
i \partial_t u =-\Delta u - u|u|^{p-1},\quad (t,x)\in \R\times \R^n
\end{equation}
We recall that NLS enjoy respectively
the conservation of the following energy:
\begin{equation}\label{eq:enls}
{\mathcal E}_{nls}(u)=\frac{1}{2}\|u\|_{\dot H^{1}(\R^n)}^2-\frac{1}{p+1}\| u  \|^{p+1}_{L^{p+1}(\R^n)},
\end{equation}
as well as the conservation of the mass,
\begin{equation}\label{massconsnls} \frac d{dt} \|u(t, x)\|_{L^2(\R^n)}^2=0
\end{equation}
and the conservation of the momentum
\begin{equation}\label{momentumnls}
P(u(x,t))= \left(-i \int_{\R^n} \bar u    \nabla  u  dx\right).
\end{equation}

It is well known that the standing wave $\psi_{\omega}(x,t)=e^{i\omega t}Q_{\omega}(x)$, where $Q_{\omega}$ is the positive ground state solution that solves
\begin{equation}\label{eq:ellit}
-\Delta w + \omega w- |w|^{p-1}w=0,
\end{equation}
is solution to \eqref{eq:mainNLS}. \\
On the other hand, thanks to the Galilean transform we know that given $\psi(t,x)$ an arbitrary  solution to \eqref{eq:mainNLS} and $v \in \R^n$ then
$$\psi_v(t,x)=\psi(t,x-vt)e^{i(\frac{v}{2} \cdot x-\frac{|v|^2}{4}t)}$$
is solution to \eqref{eq:mainNLS}. In the specific case of $\psi_{\omega}(x,t)$ this implies that
$$\psi_{Q,v}(x,t)=Q(x-vt)e^{i(\frac{v}{2} \cdot x-\frac{|v|^2}{4}t+\omega t)}$$
is the corresponding travelling wave solution to \eqref{eq:mainNLS} that moves on the line $x=vt$. \\
Moreover we notice that if $\psi=e^{i\omega t}w(x)$ is an arbitrary standing wave, namely $w$ is  solution to \eqref{eq:ellit}, then the energy of the corresponding
boosted solution $\psi_v(t,x)$ fulfills
\begin{equation}\label{eq.boostNLS}
E(\psi_v)= E(w)+\frac{|v|^2}{8}\int_{\R^n}|w|^2 dx +\frac{i}{2} \int_{\R^n} \bar w \left(v  \cdot \nabla  w \right) dx.
\end{equation}
Now we notice that the quantity $ i  \int_{\R^n} \bar w \nabla  w  dx$ corresponds  to $-P(w)$ and from the Heisenberg relations
we have the following identity
\begin{equation}\label{eq:haise}
\frac{d}{dt} <u(t), A (t)u(t)>=i <u(t), [H,A]u(t)>
\end{equation}
where $[H,A]:=HA-AH$ denotes the commutator of $A$ with $H:=-\Delta -|u|^{p-1}$ and  $u(t)$ is an arbitrary solution to \eqref{eq:mainNLS}.
This relation implies with an elementary computation that choosing $A=x$
$$\frac 12 \frac{d}{dt} <u(t), x u(t)>=P(u).$$
 Therefore, fixed $v$ and  $w$ being a standing wave for which clearly   $$\frac 12 \frac{d}{dt} <w(t), x w(t)>=0,$$ then we obtain
 $$i \int_{\R^n} \bar w \nabla  w  dx=0.$$
 This implies immediately that $E(\psi)< E(\psi_v)$.
This simple fact proves  that fixed the mass of the wave, say $||Q_{\omega}(x)||_{L^2}^2=r$, then \emph{ the least energy solution among the standing waves  $w$ and the corresponding travelling waves $\psi_v$ with $\|w\|_{L^2(\R^n)}^2=\|\psi_v\|_{L^2(\R^n)}^2=r$ is given by $Q_{\omega}(x)e^{i \omega t}$ where $\omega$ is fixed by the value of $r$}.\\

\section{Proof of Theorem \ref{maintheorem}}
\begin{lem}\label{four}
For every $f\in L^2(\R)$ there exists one unique
solution $u=A_v f\in H^1(\R)$ to the following linear problem
$$\left(\sqrt{-\Delta } u + i \left(v \partial_x \right) u\right)+(1-v)u=f.$$
Moreover we have the following bounds
\begin{equation}\label{iando0}\|A_v f\|_{L^2(\R)}\leq \alpha_v \|f\|_{H^{-1}(\R)},\end{equation}
\begin{equation}\label{iando}\|A_v f\|_{H^{1/2}(\R)}\leq \alpha_v \|f\|_{L^{5/4}(\R)}\end{equation}
where $\alpha_v=O(\frac 1{(1-v)})$ as $v\rightarrow 1$.
\end{lem}
\begin{proof}
By using the Fourier transform we have
$$\hat u(\xi)= \frac{\hat f(\xi)}{|\xi| -v\xi +1-v}.$$
By Plancharel we get
$$u\in H^1(\R) \iff (1+ |\xi|) \hat u(\xi)\in L^2(\R) \iff \frac{(1+|\xi|)\hat f(\xi)}{|\xi| -v\xi +1-v}\in L^2(\R),$$
since
$$ \frac{1+|\xi|}{|\xi| -v\xi +1-v} \leq \frac{1}{1-v}.$$

In order to prove the uniform a-priori bound \eqref{iando}
notice that by the computation above we get
$$\|A_v f\|_{H^1(\R)} \leq \alpha_v \|f\|_{L^2(\R)}$$
where
\begin{equation}\label{cv}
\alpha_v=\sup_\xi \frac {1+|\xi|}{\big| |\xi| -v\xi +1-v\big|} = \frac{1}{1-v}
\end{equation}and hence (since we are working with an operator with constant coefficients)
$$\|\langle D\rangle^{-1/2} A_v f\|_{H^1(\R)} \leq
\alpha_v \|\langle D \rangle^{-1/2}f\|_{L^2(\R)}$$
where
$\langle D\rangle^{-1/2}$ is the Fourier multiplier associated with
$(1+|\xi|^2)^{-1/4}$.
In turn it implies
$$\|A_v f\|_{H^{\frac 12}(\R)} \leq
\alpha_v \|f\|_{H^{-\frac 12}(\R)}$$
and hence we conclude since
by Sobolev embedding
$L^{\frac 54}(\R)\subset H^{-\frac 12}(\R)$. Moreover by direct computation
we get by \eqref{cv} the bound $\alpha_v=O(\frac 1{1-v})$.
\end{proof}
\begin{proof}[Proof of Theorem \ref{maintheorem}]$$$$
Let us notice that \eqref{eq:diffeq2a}
is the Euler-Lagrange equation corresponding to the critical point of the following Weinstein functional
$$W(u)=\frac{||u||_{L^5(\R)}^5}{ \left(||u||_{\dot H^{\frac 12}(\R)}^2+i \int_\R \bar u \left(v  \partial_x  u \right) dx \right)^{\frac32} ||u||_{L^2(\R)}^2}.$$
Existence of maximizers for the Weinstein functional follows arguing as in Appendix B of \cite{FL2}. Fixing $v$ let us call the maximizers
$Q_v$. Moreover arguing as in the proof of Theorem 1.1 of \cite{KLR} it is easy to show that
$$i v \int_\R \bar Q_v  \partial_x  Q_v dx <0.$$
Now from Gagliardo-Nirenberg inequalities
$$||u||_{L^5(\R)}^5\leq C_0||u||_{\dot H^{\frac 12}(\R)}^3||u||_{L^2(\R)}^2$$
$$||u||_{L^5(\R)}^5\leq C_v\left(||u||_{\dot H^{\frac 12}(\R)}^2+i v \int_\R \bar u   \partial_x  u  dx\right)^{\frac 32}||u||_{L^2(\R)}^2$$
and choosing $\psi$ having only positive Fourier components such that
$$\left(||\psi||_{\dot H^{\frac 12}(\R)}^2+i v \int_\R \bar \psi  \partial_x  \psi  dx\right)=(1-v)||\psi||_{\dot H^{\frac 12}(\R)}^2$$
we
\begin{equation}\label{eq:ii}
 C_v\geq (1-v)^{-\frac32}\frac{||\psi||_{L^5(\R)}^5}{ \left(||\psi||_{\dot H^{\frac 12}(\R)}^2+i v \int_\R \bar \psi   \partial_x  \psi dx \right)^{\frac32} ||\psi||_{L^2(\R)}^2}.
 \end{equation}
 Upper bound for $C_v$ can be obtained using the decomposition in positive and negative frequencies
 $$ \psi = \psi_+ + \psi_-, $$
 we get
 $$ \left(||\psi_\pm||_{\dot H^{\frac 12}(\R)}^2+i v \int_\R \bar \psi_\pm   \partial_x  \psi_\pm dx\right) = (1\mp v) ||\psi_\pm||_{\dot H^{\frac 12}(\R)}^2$$
 so the standard Gagliardo - Nirenberg inequality
 $$ \|\psi\|_{L^5(\R)}^5 \leq C \|\psi\|_{\dot{H}^{1/2}(\R)}^3 \|\psi\|_{L^2(\R)}^2 $$ and in the case $\psi = \psi_+$ we find
 $$ \|\psi_+\|_{L^5(\R)}^5 \leq  C(1-v)^{-\frac32} \langle (1-v)|D| \psi_+, \psi_+ \rangle^{\frac 32} \|\psi_+\|_{L^2(\R)}^2 $$
 $$ =  C(1-v)^{-\frac32} \left(||\psi_+||_{\dot H^{\frac 12}(\R)}^2+i v \int_\R \bar \psi_+  \partial_x  \psi_+  dx\right)^{3/2}\|\psi_+\|_{L^2(\R)}^2. $$
 In a similar way we find
  $$ \|\psi_-\|_{L^5(\R)}^5 \leq   C(1+v)^{-\frac32} \left(||\psi_-||_{\dot H^{\frac 12}}^2+i v \int_\R \bar \psi_-   \partial_x  \psi_-  dx\right)^{\frac 32}\|\psi_-\|_{L^2(\R)}^2 $$
 so we can conclude that
 $$  \|\psi\|_{L^5(\R)}^5 \leq C \left( \|\psi_+\|_{L^5(\R)}^5 + \|\psi_-\|_{L^5(\R)}^5 \right) \leq $$ $$ \leq C(1-v)^{-\frac32} \left(\underbrace{||\psi_+||_{\dot H^{\frac 12}(\R)}^2+i v \int_\R \bar \psi_+  \partial_x  \psi_+ dx}_{A}\right)^{\frac 32}\|\psi_+\|_{L^2(\R)}^2 + $$
 $$ + C\left(\underbrace{||\psi_-||_{\dot H^{\frac 12}}^2+i  v \int_\R \bar \psi_-   \partial_x  \psi_- dx}_{B}\right)^{\frac 32}\|\psi_-\|_{L^2(\R)}^2 \leq  $$
 $$  \leq C(1-v)^{-\frac32}  \|\psi\|_{L^2(\R)}^2 (A+B)^{\frac 32}.$$
 To this end we use the Plancherel identity and get
 $$ A+B =  \langle ( |\xi| - v \xi) \widehat{\psi_+},\widehat{\psi_+} \rangle_{L^2(\R)} + \langle ( |\xi| - v \xi) \widehat{\psi_-},\widehat{\psi_-} \rangle_{L^2(\R)} = $$
$$ = \langle ( |\xi| - v \xi) \widehat{\psi},\widehat{\psi} \rangle_{L^2(\R)} = ||\psi||_{\dot H^{\frac 12}(\R)}^2+i v \int_\R \bar \psi   \partial_x  \psi  dx.$$
 So we obtain
 $$ ||\psi||_{L^5(\R)}^5\leq C(1-v)^{-\frac32} \left(||\psi||_{\dot H^{\frac 12}(\R)}^2+i v \int_\R \bar \psi  \partial_x  \psi dx\right)^{\frac 32}||\psi||_{L^2(\R)}^2 $$
 for any $\psi \in H^{\frac 12}(\R)$ and we have
 \begin{equation}\label{eq:asy}
 C_v \lesssim (1-v)^{-\frac32}.
 \end{equation}
Now taking $Q_v$ maximizer for $W$, we
can scale $Q_v\rightarrow aQ_v(bx)$ such that $Q_v$ solves
\begin{equation}\label{eq:diffeq}
\sqrt{-\Delta }Q_v  + i v \partial_x Q_v+ (1-v)Q_v- |Q_v|^{3}Q_v=0.
\end{equation}
Notice  that $Q_v$ is now a critical point also of the following functional
$$E_{boost}(u)={\mathcal E}_{hw}(u)+\frac{1}{2}(1-v)||u||_{L^2(\R)}^2+ \frac{i}{2} \int_\R \bar u \left(v  \partial_x  u \right) dx,$$
and therefore a Pohozaev type identity yields to
\begin{equation}
\frac 12 \|Q_v\|_{\dot H^{1/2}(\R)}^2- \frac{3}{10}\| Q_v \|^{5}_{L^{5}(\R)}+\frac{i}{2}  v \int_\R \bar Q_v  \partial_x  Q_v dx=0.
\end{equation}
By the fact that $Q_v$ solves \eqref{eq:diffeq} also
$$ \|Q_v\|_{\dot H^{1/2}(\R)}^2+(1-v) \|Q_v\|_{L^2(\R)}^2- \| Q_v  \|^{5}_{L^{5}(\R)}+i  v \int_\R \bar Q_v  \partial_x  Q_v dx=0$$
we can conclude that
$$\left(||Q_v||_{\dot H^{\frac 12}(\R)}^2+i  v \int_\R \bar Q_v   \partial_x  Q_v  dx\right)= \frac{3}{2}(1-v)\|Q_v\|_{L^2(\R)}^2,$$
\begin{equation}\label{eq.imp1}
   \|Q_v\|_{L^5(\R)}^5=\frac{5}{2}(1-v)\|Q_v\|_{L^2(\R)}^2.
\end{equation}
By recalling that
$$C_v=\frac{||Q_v||_{L^5(\R)}^5}{ \left(||Q_v||_{\dot H^{\frac 12}}^2+i v \int_\R \bar Q_v   \partial_x  Q_v  dx \right)^{\frac32} ||Q_v||_{L^2(\R)}^2}$$
we get
$$C_v=\frac{5(1-v)}{2\left(||Q_v||_{\dot H^{\frac 12}}^2+i  v \int_\R \bar Q_v  \partial_x  Q_v  dx \right)^{\frac32}}$$
and
$$C_v=\frac{\frac52 (\frac23)^{\frac 32}}{(1-v)^{\frac 12}||Q_v||_{L^2(\R)}^3}.$$
Together with \eqref{eq:asy} we conclude that
\begin{equation}\label{eq:ii}
\lim_{v\rightarrow 1}||Q_v||_{L^2(\R)}=0, \ \ \ \lim_{v\rightarrow 1}\left(||Q_v||_{\dot H^{\frac 12}}^2+i v\int_\R \bar Q_v   \partial_x  Q_v dx \right)=0.
\end{equation}

We choose in Lemma \ref{four}
the forcing term $f=Q_v|Q_v|^3$ and hence by looking at the equation solved by $Q_v$ we get
$$A_v(Q_v|Q_v|^3)=Q_v$$
and hence again by Lemma \ref{four} we get
$$\|Q_v\|_{H^{1/2}(\R)}\leq \alpha_v \|Q_v |Q_v|^3\|_{L^{\frac 54}(\R)}
=\alpha_v \|Q_v\|_{L^5(\R)}^4$$
From the fact that  $C_v\geq C (1-v)^{-\frac32}$ it follows that
$$
||Q_v||_{L^2(\R)}^2=O((1-v)^{\frac 23}),   \ ||Q_v||_{L^5(\R)}^5=O((1-v)^\frac{5}{3}).
$$
Since $\alpha_v=O(\frac 1{1-v})$ then we conclude
$$\|Q_v\|_{H^{1/2}(\R)}\leq \alpha_v \|Q_v\|_{L^5(\R)}^4\leq C \frac{(1-v)^\frac{20}{15}}{(1-v)}=O((1-v)^{\frac 13}).$$

To estimate $\|Q_v\|_{H^{1}(\R)} $ we rewrite \eqref{eq:diffeq2a} in  the form
\begin{equation}\label{eq:diffeq2a1}
\sqrt{-\Delta }Q_v +  i v \partial_x Q_v = -(1-v)Q_v +  |Q_v|^{3}Q_v,
\end{equation}
and note that
$$\left\| \sqrt{-\Delta }Q_v +  i v \partial_x Q_v \right\|_{L^2(\R)} \geq  (1-v)\|Q_v\|_{\dot{H^1}(\R)}. $$
We can use the  Sobolev inequality
$$ \|Q_v\|_{L^8(\R)} \leq C \|Q_v\|_{H^{1/2}(\R)} = O((1-v)^{1/3})$$ estimating the $L^2 - $ norm of the right side of \eqref{eq:diffeq2a1} and  deducing
 $$ \|Q_v\|_{\dot{H^1}(\R)} \leq C(1-v)^{1/3}.$$

 Our final step is to take
  $0 < v_1 < v_2 < 1$ with  $$v_1 = 1 - \varepsilon, \ \ v_2 = v_1 + \delta, \ \ 0 < \delta \ll \varepsilon \ll 1,$$ and prove the inequality \eqref{eq.rel1}.
  For the purpose, we set
  $$ R = Q_{v_1} - Q_{v_2} $$ so that it is   solution to
  \begin{eqnarray}\label{eq.difR1}
   & \sqrt{-\Delta} R + (1-v_1)R + i v_1 \partial_x R = \\ \nonumber = &\underbrace{ -\delta Q_{v_2} + i \delta \partial_x Q_{v_2} + \left( \frac{Q_{v_1}|Q_{v_1}|^3 - Q_{v_2}|Q_{v_2}|^3}{Q_{v_1}-Q_{v_2}}\right) R}_{F(Q_{v_1}, Q_{v_2})}.
  \end{eqnarray}
  By using Lemma \ref{four} we rewrite the above equation as
  $$ R = A_{v_1} \left(F(Q_{v_1}, Q_{v_2}) \right)$$
  so taking the $L^2 -$ norm, we can write
  $$ \|R\|_{L^2} \leq \|A_{v_1}(\delta Q_{v_2})\|_{L^2} +  \|A_{v_1}(\delta \partial_x Q_{v_2})\|_{L^2(\R)} +$$
  $$+ \left\| A_{v_1} \left[\left( \frac{Q_{v_1}|Q_{v_1}|^3 - Q_{v_2}|Q_{v_2}|^3}{Q_{v_1}-Q_{v_2}}\right) R\right] \right\|_{L^2(\R)} .$$
     and using \eqref{iando0} we get the inequalities
  $$ \|A_{v_1} \left(\delta Q_{v_2}\right)\|_{L^2(\R)} \leq C \delta \| Q_{v_2}\|_{L^2(\R)}  \leq C \delta (1- v_1)^{1/3} $$
  and
  $$\|A_{v_1}(\delta \partial_x Q_{v_2})\|_{L^2(\R)} = \|A_{v_1}(\delta  Q_{v_2})\|_{H^1(\R)} \leq C \delta \| Q_{v_2}\|_{L^2(\R)}  \leq C \delta (1- v_1)^{1/3}. $$
  Further, we use the pointwise bound
  $$ \left| \frac{Q_{v_1}|Q_{v_1}|^3 - Q_{v_2}|Q_{v_2}|^3}{Q_{v_1}-Q_{v_2}}\right| \leq C \left( |Q_{v_1}|^3 + |Q_{v_2}|^3 \right).$$
Using \eqref{iando0}, we find
$$ \left\| A_{v_1}  \left[\left( \frac{Q_{v_1}|Q_{v_1}|^3 - Q_{v_2}|Q_{v_2}|^3}{Q_{v_1}-Q_{v_2}}\right) R\right] \right\|_{L^2(\R)}  \leq $$
$$ \leq  C \left\|  \left[\left( \frac{Q_{v_1}|Q_{v_1}|^3 - Q_{v_2}|Q_{v_2}|^3}{Q_{v_1}-Q_{v_2}}\right) R\right] \right\|_{H^{-1}(\R)} \leq $$
$$ \leq C \left\|  \left[\left( \frac{Q_{v_1}|Q_{v_1}|^3 - Q_{v_2}|Q_{v_2}|^3}{Q_{v_1}-Q_{v_2}}\right) R\right] \right\|_{L^1(\R)} \leq $$
$$ \leq C  \left\|  \left( |Q_{v_1}|^3 + |Q_{v_2}|^3 \right) R \right\|_{L^1(\R)} \leq $$
$$ \leq C \left(\|Q_{v_1}\|_{L^6(\R)}^3 + \|Q_{v_2}\|_{L^6(\R)}^3 \right) \|R\|_{L^2(\R)}.$$
Applying the Sobolev embedding once more, we get
$$  \|Q_{v_1}\|_{L^6(\R)}^3 + \|Q_{v_2}\|_{L^6(\R)}^3 \leq C \left( \|Q_{v_1}\|_{H^{1/2}(\R)}^3 + \|Q_{v_2}\|_{H^{1/2}(\R)}^3\right) \leq C(1-v_1). $$
Summing up, the above estimates lead to the inequality
 $$ \|R\|_{L^2(\R)} \leq C \delta (1- v_1)^{1/3} + C (1-v_1)\|R\|_{L^2(\R)}$$
 so choosing $1-v_1 = \varepsilon$ so small that $C\varepsilon < 1/2,$ we find
 $$ \|R\|_{L^2(\R)} \leq C \delta (1- v_1)^{1/3}$$
 and this completes the proof of \eqref{eq.rel1}.

Now we show that small data scattering does not occur. More precisely we will show that if the wave operators exists they are not continuous.
For any $0<v_1<v_2<1$ let  $\psi_j(t,x) = e^{\im (1-v_j)t}Q_{v_j}(x-v_jt)$ where  $Q_{v_1}, Q_{v_2} \in  H^{\frac{1}{2}}(\R)$ are the traveling waves constructed above. Then we show that 
\begin{equation}\label{Pit1}
||\psi_1(t) - \psi_2(t)||^2_{L^2(\R)}\sim ||\psi_1(t)||^2_{L^2(\R)}  + ||\psi_2(t)||^2_{L^2(\R)} .
\end{equation}

We can show that for any $\varepsilon>0$  and any $f,g \in L^2(\R)$ with $\|f\|_{L^2(\R)}= \|g\|_{L^2(\R)}$ exists $\tau_0 = \tau(\varepsilon, f, g)$ so that
\begin{equation}\label{eq.Pit2}
   \left|\int_{\R} \overline{f(x)} g(x+\tau) dx \right| \leq \varepsilon, \ \ \forall \tau \geq \tau_0.
\end{equation}
Indeed, this is a consequence of the fact that the weak limit of $\{g(\cdot + \tau ) \}_{\tau \to \infty} $ has to be zero.
The relation \eqref{Pit1}  follows from
$$ \left| e^{\im (1-v_1)t} e^{-\im (1-v_2)t} \int_{\R} \overline{Q_{v_1} (x-v_1 t)}  Q_{v_2} (x-v_2 t) dx \right| = $$ $$ = \left| \int_{\R} \overline{Q_{v_1} (y)}  Q_{v_2} (y-(v_2-v_1) t) dy \right|  $$

so choosing
$$ f(x) = \frac{Q_{v_1}(x)}{\|Q_{v_1}\|_{L^2(\R)}}, \ g(x) = \frac{Q_{v_2}(x)}{\|Q_{v_2}\|_{L^2(\R)}}$$
and applying \eqref{eq.Pit2} we get
$$ ||\psi_1(t) - \psi_2(t)||^2_{L^2(\R)} = ||\psi_1(t)||^2_{L^2(\R)}  + ||\psi_2(t)||^2_{L^2(\R)} - 2 {\rm Re } \langle \psi_1(t), \psi_2(t) \rangle  \geq  $$ $$ \geq ||\psi_1(t)||^2_{L^2(\R)}  + ||\psi_2(t)||^2_{L^2(\R)} - \varepsilon  ||\psi_1(t)||_{L^2(\R)}||\psi_2(t)||_{L^2(\R)} \geq  $$ $$ \geq \frac{1}{2}||\psi_1(t)||^2_{L^2(\R)}  + \frac{1}{2}||\psi_2(t)||^2_{L^2(\R)}.$$
and this
completes the proof. If the wave operators exist and they are $L^2$ bounded, then
$$ \psi_1(0,x)=\psi_1(x)  \ \ \to  W(\psi_1(x))= \phi_1, \psi_2(0,x)=\psi_2(x)  \ \ \to  W(\psi_2(x))= \phi_2$$ satisfy
$$ \lim_{t \to \infty}\|\psi_1(t) - e^{\im t |D|} \phi_1 \|_{L^2(\R)} =  \lim_{t \to \infty}\|\psi_2(t) - e^{\im t |D|} \phi_2 \|_{L^2(\R)} = 0$$
and since
$$ \|e^{\im t |D|} \phi_1 - e^{\im t |D|} \phi_2 \|_{L^2(\R)} = \| \phi_1 -  \phi_2 \|_{L^2(\R)},$$
we would have
\begin{equation}\label{eq.Pit3}
   \|\psi_1(t)- \psi_2(t) \|_{L^2(\R)} \leq C \| \phi_1 -  \phi_2 \|_{L^2(\R)}.
\end{equation}
From this estimate we easily get a contradiction referring to \eqref{Pit1} and choosing $v_2  = v_1 + \delta, \ \delta \ll 1-v_1.$
Indeed,  using \eqref{eq.rel1}  we can write
$$ \|Q_{v_1} - Q_{v_2}\|_{L^2(\R)} \leq C \delta (1-v_1)^{\frac 13},$$
and
$$ \|\psi_1(t)- \psi_2(t) \|^2_{L^2(\R)} \sim \|\psi_1(t) \|^2_{L^2} + \|\psi_2(t) \|^2_{L^2(\R)} = \|Q_{v_1} \|_{L^2(\R)}^2 + \|Q_{v_2} \|_{L^2(\R)}^2 \geq  C (1-v_1)^{2/3}.$$
Choosing
$$ (1-v_1)^{2/3} \gg \delta^2 (1-v_1)^{2/3} \ \  \Longleftrightarrow \ \  \delta \ll 1, $$ we get
$$  \|\psi_1(t)- \psi_2(t) \|^2_{L^2(\R)}  \gg \|Q_{v_1} - Q_{v_2}\|^2_{L^2(\R)}$$
and this contradicts \eqref{eq.Pit3}.

\end{proof}
\section{Nonexistence  of traveling waves moving at the speed of light}
\begin{lem}[Spatial decay for traveling waves for HW]\label{lem: decay}
Let $v=1$ and $Q_1 \in H^{\frac 12}(\R)$ be a  solution to
\begin{equation}\label{eq:diffeq3}
\sqrt{-\Delta }Q_1 + i  \partial_x Q_1+\omega Q_1- |Q_1|^{3}Q_1=0,
\end{equation}
then
$$|Q_1(x)|\leq \frac{C}{1+|x|^2}. $$
\end{lem}
\begin{proof}
It is evident that if $Q_1$ is solution to \eqref{eq:diffeq3} then  $Q_1 \in H^1(\R)$ and therefore $\lim_{|x|\rightarrow \infty}|Q_1(x)|=0$ and $Q_1$ is bounded. 
The decay for traveling waves moving with speed $|v|<1$ has been proved in \cite{GLPR}. Here we prove that if traveling waves moving at the speed of light exist then they shall fulfill a certain asymptotics. By scaling argument it is clear that $\omega>0$, see e.g. \eqref{eq.imp1}. For semplicity and by scaling we consider \eqref{eq:diffeq3} in the case $\omega=1$. We recall that an harmonic extension for a smooth function $f:\R\rightarrow \C$ is given by a function
 $F:\R^2_{+}=\R\times [0, \infty) \rightarrow \C$
  fulfilling
\begin{equation}\label{eq:ext0}
\partial_x^2 F+\partial_y^2 F=0 \text{ on } \R^2_{+} \ \ \ F(x,0)=f(x).
\end{equation}

It is elementary to notice (using that $F$ is harmonic in $\R^2_{+}$ )that
$$-\partial_y F(x,0)=\sqrt{-\Delta }f(x).$$

Now, given $Q_1$  solution to \eqref{eq:diffeq3} with $\omega=1$ we consider  $q_1(x,y)$ its extension to $\R^2_{+}$ fulfilling
\begin{equation}\label{eq:ext7}
\partial_x^2 q_1+\partial_y^2 q_1=0 \text{ on } \R^2_{+}
\end{equation}
\begin{equation}\label{eq:ext8}
\partial_y q_1-i   \partial_x q_1 =  q_1- |q_1|^{3}q_1 \text{ on } (x,0)
\end{equation}
\begin{equation}\label{eq:ext9}
\lim_{|x| \rightarrow \infty} |q_1(x,0)|=0.
\end{equation}
Our idea it to deduce the decay of $Q_1$ by looking at the decay of the solution $q_1(x,y)$ to the extension problem. The case $v=0$ it has been studied by Amick-Toland \cite{AT} and Kenig-Martel-Robbiano \cite{KMR}.
Consider hence the boundary value problem

\begin{equation}\label{eq:ext4}
\partial_x^2 w+\partial_y^2 w=0 \text{ on } \R^2_{+}
\end{equation}
\begin{equation}\label{eq:ext5}
w-\partial_y w+i v  \partial_x w =  f \text{ on } (x,0)
\end{equation}
\begin{equation}\label{eq:ext6}
\lim_{|x| \rightarrow \infty} |w(x)|=0.
\end{equation}
The crucial point is to show the existence of a function $G(x,y)$ such
that
\begin{enumerate}
\item $G$ is harmonic on $\R^2_{+}$
\item $G-\partial_y G+i   \partial_x G =\frac{1}{\pi}\frac{y}{x^2+y^2}\equiv \frac{1}{y} b_0 \left( \frac{x}{y} \right),$
\end{enumerate}
where
$$ b_0(x)= \frac{1}{\pi(1+x^2)}.$$
Indeed, if $G$ fulfills (1)-(2) then
\begin{equation}\label{eq:decadcr}
w(x,y)=\int_{\R} G(x-t,y)f(t)dt
\end{equation}
is solution to \eqref{eq:ext4}, \eqref{eq:ext5}, \eqref{eq:ext6}.\\
Concerning \eqref{eq:ext7}, \eqref{eq:ext8}, \eqref{eq:ext9}, by choosing $f=-|q_1|^{3}q_1$ we deduce from \eqref{eq:decadcr} that $q_v$ fulfills
\begin{equation}\label{eq:crucrelaz}
q_1(x,y)=\int_{\R} G(x-t,y)|q_1|^{3}q_1(t)dt.
\end{equation}
From this last estimate we shall deduce the decay of $q_v$ using the information about the decay of $G$.\\
In the  case $v=0$ studied in \cite{AT} the function
$G$ is explicitly given by
$$G(x,y)=\frac{1}{\pi}\int_0^{\infty}\frac{e^{-s}(y+s)}{x^2+(y+s)^2}ds.$$
Now if we apply inverse Fourier transform in the variable $x$ we get
\begin{enumerate}
\item $\partial_y^2 \hat G(\xi, y)-\xi^2 \hat G(\xi, y)=0$ on $\R^2_{+}$
\item $\left. \left( \hat G-\partial_y \hat G + \xi \hat G \right) \right\vert_{y=0} = 1,$
\end{enumerate}
We look for solution in the form
$$ \hat G(\xi, y) = e^{-y|\xi|} C(\xi),$$
where $C(\xi)$ is defined by the boundary condition 
$$\left. \left( \hat G-\partial_y \hat G + \xi \hat G \right) \right\vert_{y=0} = 1,$$
i.e.
$$ C(\xi) = \frac{1}{1+|\xi|+\xi}$$
Hence
$$ G(x,y) = \int_{-\infty}^\infty e^{-y|\xi|} e^{-2 \pi \mathrm{i} x \xi} \frac{d\xi}{(1+|\xi|+\xi)} $$

We have
$$G(x,y) =\underbrace{ \int_{-\infty}^0e^{y \xi} e^{-2 \pi \mathrm{i} x \xi} d\xi}_{I_1}+ \underbrace{ \int_{0}^\infty e^{-y\xi} e^{-2 \pi \mathrm{i} x \xi} \frac{d\xi}{(1+2\xi)}}_{I_2}.$$
By fundamental theorem of calculus we obtain immediately that
\begin{equation}\label{eq:pezneg}
I_1=\left(\frac{1}{y-2\pi i x} \right)\int_{-\infty}^0\frac{d}{d\xi}\left(e^{y \xi} e^{-2 \pi \mathrm{i} x \xi}\right) d\xi=\frac{1}{y-2\pi i x}.
\end{equation}
By integration by parts we get
\begin{equation*}
I_2=\left(\frac{1}{-y-2\pi i x} \right)\int_{0}^{\infty}\frac{d}{d\xi}\left(e^{-y \xi} e^{-2 \pi \mathrm{i} x \xi}\right)  \frac{d\xi}{(1+2\xi)}
\end{equation*}
\begin{equation}\label{eq:pezpos}
I_2=\left(\frac{1}{y+2\pi i x} \right)+ \left(\frac{2}{-y-2\pi i x} \right) \int_{0}^{\infty}e^{-y \xi} e^{-2 \pi \mathrm{i} x \xi}  \frac{d\xi}{(1+2\xi)^2}
\end{equation}
such that we obtain
\begin{equation}\label{eq:peztut}
G(x,y)=\left(\frac{2y}{y^2+4\pi^2  x^2} \right)+ \left(\frac{2}{-y-2\pi i x} \right) \int_{0}^{\infty}e^{-y \xi} e^{-2 \pi \mathrm{i} x \xi}  \frac{d\xi}{(1+2\xi)^2}
\end{equation}
that thanks to a second integration by parts can be estimated as
\begin{equation}\label{eq:peztut2}
G(x,y)=\left(\frac{2y}{y^2+4\pi^2  x^2} \right)+ \left(\frac{2}{(y+2\pi i   x)^2} \right)+ \left(\frac{8}{(y+2\pi i   x)^2} \right) \int_{0}^{\infty}e^{-y \xi} e^{-2 \pi \mathrm{i} x \xi}  \frac{d\xi}{(1+2\xi)^3}.
\end{equation}
Eq. \eqref{eq:peztut2} implies the following decay
\begin{equation}\label{eq:peztut3}
|G(x,y)|\leq C \left(\frac{1+y}{y^2+4\pi^2  x^2} \right)
\end{equation}
Now recalling \eqref{eq:crucrelaz}
$$q_1(x,y)=\int_{\R} G(x-t,y)|q_1|^{3}q_1(t)dt,$$
and  following verbatim the proof at pag. 24 in \cite{AT2} we obtain 
$$|Q_1(x)|=|q_1(x,0)|\leq \frac{C}{1+|x|^2}$$
and hence the desired decay. We give a brief sketch of the argument of \cite{AT2} for reader's convenience.\\
Given $\delta>0$, let $X(\delta)>0$ such that $|u(x)|<\delta$ if $|x|>X(\delta)$. Now, let us define, for $\alpha=0$ or $\alpha=2$, the Banach space $C_{\delta}^{\alpha}$ defined by continuous functions defined on $\{ x\in \R, \ \ |x|\geq X(\delta) \}$ such that
$$||f||_{\alpha, \delta}=\sup \{ (1+|x|^{\alpha})|f(x)|, \ \ |x|\geq X(\delta) \}<+\infty.$$
Now thanks to \eqref{eq:crucrelaz} we have
$$q_1(x,0)=\underbrace{\int_{|x|<X(\delta)} G(x-t,0)|q_1|^{3}q_1(t)dt}_{A_{\delta}(x)}+\int_{|x|\geq X(\delta)} G(x-t,0)|q_1|^{3}q_1(t)dt.$$
Defining the operator $T_{\delta}$ defined on $C_{\delta}^{\alpha}$ as
$$T_{\delta}g(x)=\int_{|x|\geq X(\delta)} G(x-t,0)|q_1(t)|^{3}g(t)dt,$$
we have the elementary estimates
$$|A_{\delta}(x)|\leq const\left(\frac{1}{1+x^2} \right), \ \ |T_{\delta}g(x)|\leq const \cdot \delta^3  \left(\frac{||g||_{\alpha, \delta}}{1+|x|^{\alpha}} \right)$$
By the fact that $q_1(x,0)\neq 0$ we can choose $\delta$ sufficently small such that $A_{\delta}\neq 0$. Hence, by the contraction mapping principle, the equation $g=A_{\delta}+T_{\delta}g$ has a unique solution in $C_{\delta}^{\alpha}$ for $\delta$ sufficently small.
By the fact that $q_1(x,0)$ fulfills \eqref{eq:crucrelaz} we get the desired decay.
\end{proof}

\begin{proof}[Proof of Theorem \ref{thm:nofast}]$$$$
By Heisenberg relations we have
$$\frac{d}{dt} <u, A u>=i <u,[\mathcal{H},A]u>$$
where $[\mathcal{H},A]=\mathcal{H}A-A\mathcal{H}$ denotes the commutator and $\mathcal{H}$ is the time dependent operator $\sqrt{-\Delta} -|u|^{3}$.\\
We recall  the definition of the Hilbert transform $H$ defined
$$ H(f)(x) = \int Pv \left( \frac{1}{x-y} \right) f(y) dy = \int_0^\infty \frac{f(x-y)-f(x+y)}{y} \ dy,$$
and fulfilling in Fourier variables
$\widehat{H f}(\xi)=-i \pi sign(\xi)\hat{f}(\xi)$.
We shall use the following commutator relations
\begin{equation}\label{eq.CR1a1}
    [AB,C]=A[B,C]+[A,C]B.
\end{equation}
and
\begin{alignat}{2}
\nonumber & [x,D_x]= \frac{1}{\im} [x,\partial_x] = \im, \\
\label{eq.HTR2}
    & [x,H](f)(x) = \int (x-y) Pv \left(\frac{1}{x-y} \right) f(y) dy = \int f(y) dy, \\
    \nonumber &\partial_x[x,H] = [x,H] \partial_x = 0,  \\ \nonumber
    & [x,\sqrt{-\Delta}] = \frac{1}{\pi}[x, \partial_x H] = \frac{1}{\pi} [x,\partial_x]H + \underbrace{\frac{1}{\pi} [x,H]\partial_x}_{= 0} =  -\frac{H}{\pi}.
\end{alignat}
Therefore from Heisenberg relation and using commutator relations  we get
$$\frac{d}{dt} \int_{\R}x |u(x,t)|^2 dx =i <u,[\sqrt{-\Delta} ,x]u>$$
and  we conclude thanks to Cauchy-Schwarz and Plancherel that
\begin{equation}\label{eq:meangrowth}
\left| \frac{d}{dt} \int_{\R}x |u(x,t)|^2 dx \right| = \left| <u, \frac{H(u)}{\pi}> \right | \leq   \frac{1}{\pi} ||u||_{L^2(\R)}||H(u)||_{L^2(\R)}=||u||_{L^2(\R)}^2.
\end{equation}
By applying \eqref{eq:meangrowth} to $u(x,t)=e^{i \omega t}Q_{v}(x-vt)$  and thanks to Lemma \ref{lem: decay} we get
$$\left| \frac{d}{dt} \int_{\R}x  |Q_{v}(x-vt)|^2 dx \right|=\left| v  \int_{\R} |Q_{v}(x)|^2 dx \right|  \leq  ||Q_v ||_{L^2(\R)}^2,$$
and hence $$|v|\leq 1.$$
By observing that equality in Cauchy-Schwarz means that $u$ and $\frac{H(u)}{\pi}$ are parallel we deduce that if a travelling wave moves at speed $|v|=1$,  i.e. at the speed of light, therefore supp $\hat Q_v\subset (-\infty, 0]$ or supp $\hat Q_v \subset [0, \infty)$.\\
By elementary scaling arguments it is evident that when $v=1$ then  supp $\hat Q_v\subset (-\infty, 0]$ and when  $v=-1$ then  supp 
 $\hat Q_v \subset [0, \infty)$.   
 Let us notice that the equation fulfilled by the traveling wave moving at the speed of light
 $$\sqrt{-\Delta }Q_1 + i  \partial_x Q_1+Q_1- |Q_1|^{3}Q_1=0,$$
 thanks to the fact that  supp $\hat Q_1\subset (-\infty, 0]$, can be rewritten as
\begin{equation}\label{eq:ris}
-2\partial_x Q_1+i Q_1=i |Q_1|^3Q_1,
\end{equation}
 \begin{equation}\label{eq:risc}
2\partial_x \bar Q_1+i \bar Q_1=i |Q_1|^3\bar Q_1.
 \end{equation}
Now thanks to \eqref{eq:ris} and \eqref{eq:risc} we conclude that
$$\partial_x \left(|Q_1|^2\right) =\partial_x (Q_1)\bar Q_1 +\partial_x (\bar Q_1) Q_1=0,$$
which imples that  $Q_1(x)=c$ and hence $Q_1\equiv 0.$
 
\end{proof}

\end{document}